\documentclass[11pt, reqno]{amsart}
\usepackage[dvipsnames]{xcolor}
\usepackage{graphicx}
\usepackage{wrapfig}
\usepackage{caption}
\usepackage{array}
\usepackage{amssymb}
\usepackage{amsthm}
\usepackage{amsmath}
\usepackage{verbatim}
\usepackage{color,soul}
\usepackage{hhline}
\usepackage{tikz}
\usepackage[shortlabels]{enumitem}
\usepackage{diagbox}
\usetikzlibrary{quotes,angles,positioning}
\usepackage{lscape}
\usetikzlibrary{arrows.meta}
\usepackage[normalem]{ulem}
\usepackage{mathtools, nccmath}
\usepackage[shortlabels]{enumitem}
\usepackage[bordercolor=gray!20,backgroundcolor=blue!10,linecolor=blue!50,textsize=footnotesize,textwidth=0.8in]{todonotes}
\usepackage{xfrac}

\newtheorem{thm}{Theorem}[section]

\newtheorem{corollary}[thm]{Corollary}

\newtheorem{lemma}[thm]{Lemma}

\newtheorem{proposition}[thm]{Proposition}
\newtheorem{claim}[thm]{Claim}
\newtheorem{theorem}[thm]{Theorem}

\theoremstyle{definition}

\newtheorem{definition}[thm]{Definition}

\newtheorem{example}[thm]{Example}

\newtheorem{conjecture}[thm]{Conjecture}

\newtheorem{remark}[thm]{Remark}

\numberwithin{equation}{section}
\newenvironment{acknowledgements}{%
  \begin{abstract}
}{%
  \end{abstract}
}

\newcommand{\rk}{{\rm rk}}
\newcommand{\Cf}{{\tt CnFct}(B_3)}
\newcommand{\CF}{{\tt CnFct}(B_4)}
\newcommand{\CFn}{{\tt CnFct}(B_n)}

\newcommand{\LCF}{{\tt LCF}}
\newcommand{\BG}{{\tt BGen}}
\newcommand{\red}{{\tt red}}
\newcommand{\Red}{{\tt Red}}
\newcommand{\s}{{\tt SS}}
\newcommand{\SSS}{{\tt SSS}}
\newcommand{\USS}{{\tt USS}}

\newcommand{\SL}{{\tt SL}}
\newcommand{\sel}{{\tt sl}}
\newcommand{\D}{{\mathcal D}}
\newcommand{\nb}{{\tt nb}}

\providecommand\st{}
\newcommand\RelSymbol[1][]{%
\nonscript\:#1\vert
\allowbreak
\nonscript\:
\mathopen{}}
\DeclarePairedDelimiterX\GenRels[1]\langle\rangle{%
\renewcommand\st{\RelSymbol[\delimsize]}
#1}

\newcommand{\tik}[1]{%
  \begin{tikzpicture}[baseline=-\dimexpr\fontdimen22\textfont2\relax]
  #1
  \end{tikzpicture}%
}

\newcommand{\aone}{%
  \tik{
    \coordinate (1) at (0,-.15);
\coordinate (2) at (.25,-.15);
\coordinate (3) at (.25,.1);
\coordinate (4) at (0,.1);
\foreach \i/\Position in {1/below, 2/below, 3/above, 4/above}{
    \fill (\i) circle (1pt) node [\Position] {};}
    \draw (1)--(2)--cycle;}%
}

\newcommand{\atwo}{%
  \tik{
    \coordinate (1) at (0,-.15);
\coordinate (2) at (.25,-.15);
\coordinate (3) at (.25,.1);
\coordinate (4) at (0,.1);
\foreach \i/\Position in {1/below, 2/below, 3/above, 4/above}{
    \fill (\i) circle (1pt) node [\Position] {};}
    \draw (2)--(3)--cycle;}%
}
\newcommand{\athree}{%
  \tik{
    \coordinate (1) at (0,-.15);
\coordinate (2) at (.25,-.15);
\coordinate (3) at (.25,.1);
\coordinate (4) at (0,.1);
\foreach \i/\Position in {1/below, 2/below, 3/above, 4/above}{
    \fill (\i) circle (1pt) node [\Position] {};}
    \draw (3)--(4)--cycle;}%
}

\newcommand{\afour}{%
  \tik{
    \coordinate (1) at (0,-.15);
\coordinate (2) at (.25,-.15);
\coordinate (3) at (.25,.1);
\coordinate (4) at (0,.1);
\foreach \i/\Position in {1/below, 2/below, 3/above, 4/above}{
    \fill (\i) circle (1pt) node [\Position] {};}
    \draw (4)--(1)--cycle;}%
}

\newcommand{\e}{%
  \tik{
    \coordinate (1) at (0,-.15);
\coordinate (2) at (.25,-.15);
\coordinate (3) at (.25,.1);
\coordinate (4) at (0,.1);
\foreach \i/\Position in {1/below, 2/below, 3/above, 4/above}{
    \fill (\i) circle (1pt) node [\Position] {};}}%
}
\newcommand{\bone}{%
  \tik{
    \coordinate (1) at (0,-.15);
\coordinate (2) at (.25,-.15);
\coordinate (3) at (.25,.1);
\coordinate (4) at (0,.1);
\foreach \i/\Position in {1/below, 2/below, 3/above, 4/above}{
    \fill (\i) circle (1pt) node [\Position] {};}
    \draw (1)--(3)--cycle;}%
}

\newcommand{\btwo}{%
  \tik{
    \coordinate (1) at (0,-.15);
\coordinate (2) at (.25,-.15);
\coordinate (3) at (.25,.1);
\coordinate (4) at (0,.1);
\foreach \i/\Position in {1/below, 2/below, 3/above, 4/above}{
    \fill (\i) circle (1pt) node [\Position] {};}
    \draw (2)--(4)--cycle;}%
}
\newcommand{\atwoaone}{%
  \tik{
    \coordinate (1) at (0,-.15);
\coordinate (2) at (.25,-.15);
\coordinate (3) at (.25,.1);
\coordinate (4) at (0,.1);
\foreach \i/\Position in {1/below, 2/below, 3/above, 4/above}{
    \fill (\i) circle (1pt) node [\Position] {};}
    \draw (1)--(2)--(3)--(1)--cycle;}%
}

\newcommand{\athreeatwo}{%
  \tik{
    \coordinate (1) at (0,-.15);
\coordinate (2) at (.25,-.15);
\coordinate (3) at (.25,.1);
\coordinate (4) at (0,.1);
\foreach \i/\Position in {1/below, 2/below, 3/above, 4/above}{
    \fill (\i) circle (1pt) node [\Position] {};}
    \draw (2)--(3)--(4)--(2)--cycle;}%
}
\newcommand{\afourathree}{%
  \tik{
    \coordinate (1) at (0,-.15);
\coordinate (2) at (.25,-.15);
\coordinate (3) at (.25,.1);
\coordinate (4) at (0,.1);
\foreach \i/\Position in {1/below, 2/below, 3/above, 4/above}{
    \fill (\i) circle (1pt) node [\Position] {};}
    \draw (1)--(3)--(4)--(1)--cycle;}%
}

\newcommand{\aoneafour}{%
  \tik{
    \coordinate (1) at (0,-.15);
\coordinate (2) at (.25,-.15);
\coordinate (3) at (.25,.1);
\coordinate (4) at (0,.1);
\foreach \i/\Position in {1/below, 2/below, 3/above, 4/above}{
    \fill (\i) circle (1pt) node [\Position] {};}
    \draw (1)--(2)--(4)--(1)--cycle;}%
}

\newcommand{\aoneathree}{%
  \tik{
    \coordinate (1) at (0,-.15);
\coordinate (2) at (.25,-.15);
\coordinate (3) at (.25,.1);
\coordinate (4) at (0,.1);
\foreach \i/\Position in {1/below, 2/below, 3/above, 4/above}{
    \fill (\i) circle (1pt) node [\Position] {};}
    \draw (1)--(2)--cycle;
    \draw (3)--(4)--cycle;}%
}

\newcommand{\atwoafour}{%
  \tik{
    \coordinate (1) at (0,-.15);
\coordinate (2) at (.25,-.15);
\coordinate (3) at (.25,.1);
\coordinate (4) at (0,.1);
\foreach \i/\Position in {1/below, 2/below, 3/above, 4/above}{
    \fill (\i) circle (1pt) node [\Position] {};}
    \draw (2)--(3)--cycle;
    \draw (1)--(4)--cycle;}%
}
\newcommand{\delt}{%
  \tik{
    \coordinate (1) at (0,-.15);
\coordinate (2) at (.25,-.15);
\coordinate (3) at (.25,.1);
\coordinate (4) at (0,.1);
\foreach \i/\Position in {1/below, 2/below, 3/above, 4/above}{
    \fill (\i) circle (1pt) node [\Position] {};}
    \draw (1)--(2)--(3)--(4)--cycle;}%
}

\begin{document}
\title{On the negative band number}
\author{Michele Capovilla-Searle}
\address{Department of Mathematics, University of Iowa, Iowa City, IA 52242}
\email{michele-capovilla-searle@uiowa.edu}

\author{Tetsuya Ito}
\address{Department of Mathematics, Graduate School of Science, Kyoto University, Kyoto 606-8502, Japan}
\email{tetitoh@math.kyoto-u.ac.jp}

\author{Keiko Kawamuro}
\address{Department of Mathematics, University of Iowa, Iowa City, IA 52242}
\email{keiko-kawamuro@uiowa.edu}

\author{Rebecca Sorsen}
\address{Department of Mathematics, College of St. Mary, Omaha, NE 68106-2377}
\email{rsorsen@csm.edu}

%\date{\today}
\maketitle
\newcolumntype{P}[1]{>{\centering\arraybackslash}p{#1}}

\begin{abstract}
We study the negative band number of braids, knots, and links using Birman, Ko, and Lee's left-canonical form of a braid. 
As applications,  we characterize up to conjugacy strongly quasipositive braids and almost strongly quasipositive braids. 
%We also study behavior of the negative band number under braid stabilization operation. 
\end{abstract}

\section{Introduction}

The Artin braid group $B_n$  \cite{Artin} is defined by the presentation: 
\[
B_n= \GenRels*{ \sigma_1,\dots,\sigma_{n-1} \st
\begin{medsize}
\begin{aligned} & \sigma_i\sigma_j=\sigma_j\sigma_i,\enspace |i-j|>1 \\[-0.5ex] %
 & \sigma_i\sigma_{i+1}\sigma_i=\sigma_{i+1}\sigma_i\sigma_{i+1}, \enspace i=1,\dots,n-2%
\end{aligned}
\end{medsize}
} 
\]
If we view an element of $B_n$ as braided $n$ strands, the Artin generator $\sigma_i$ exchanges  adjacent $i$th and $i+1$st strands. In this paper we adopt the so called band generators (or Birman-Ko-Lee generators \cite{Birman}) denoted $a_{ij}$, which exchanges $i$th and $j$th strands, as defined in Section \ref{Sec2}. 
We call $a_{ij}$ a {\em positive} band and $a_{ij}^{-1}$ a {\em negative} band. 
We introduce the {\em negative band number} as follows: 
\begin{definition}\label{def-of-m}
For an $n$-braid word $W$ in band generators the negative band number $\nb(W)$ is the number of negative bands that appear in the word $W$. 
For a braid $\beta\in B_n$ we define the negative band number by
$$
\nb(\beta)=\min\left\{ \nb(W) \, \middle\vert \ 
W \text{ represents } \beta
\right\}.
$$
Similarly for the conjugacy class $[\beta]$ of $\beta$ we define:
$$
\nb([\beta])=\min\{\nb(\beta') \mid \text{ $\beta'$ is conjugate to $\beta$ } \}
$$
For a knot or link $K \subset S^3$ we define:
$$
\nb(K)=\min\left\{ \nb(\beta) \mid 
\text{ $\beta$ is a braid representative of $K$ }
\right\}
$$
By the definition we have $0\leq \nb(K) \leq \nb([\beta]) \leq \nb(\beta)\leq \nb(W).$
\end{definition}

Counting the number of negative bands gives a filtration on the braid group: 
Define ${\mathcal F}_i \subset B_n$ to be the set of braids $\beta \in B_n$ with $\nb(\beta)\leq i$. 
Let $P=P(n)\subset B_n$ be the monoid of positive braids in the Artin generators. 
We have ${\mathcal F}_j \cdot {\mathcal F}_k \subset {\mathcal F}_{j+k}$ and a filtration 
$$P \subset {\mathcal F}_0 \subset {\mathcal F}_1 \subset {\mathcal F}_2 \subset {\mathcal F}_3 \subset \cdots \subset \bigcup_{i=0}^\infty {\mathcal F}_i=B_n.$$
The set $\mathcal F_0$ forms a monoid in $B_n$. Elements of $\mathcal F_0$ are called {\em strongly quasipositive} (SQP) braids, originally defined and studied intensively by Rudolph.
We call elements of $\mathcal F_1$ {\em almost strongly quasipositive} (ASQP) braids, and elements in $\mathcal F_1 \setminus \mathcal F_0$ {\em strictly} ASQP braids. 
We also note that in \cite{Etnyre-VHMorris} Etnyre and VanHorn Morris study inclusions of monoids $P(n)\subset SQP(n) =\mathcal F_0 \subset QP(n) \subset B_n$, where $QP(n)$ is the monoid of quasipositive $n$-braids.

Our first goal is to characterize SQP and ASQP braids in terms of the Birman-Ko-Lee left-canonical form. 

Recall that the BKL left-canonical form gave a polynomial-time solution to the word problem and the conjugacy problem: Given two braid words $W$ and $W'$ determining whether $W=W'$ in $B_n$ is called the word problem, and determining whether $W$ and $W'$ are conjugate in $B_n$, is called the conjugacy problem. These problems have been studied and solved by a number of people, including Artin \cite{Artin}, Garside \cite{Garside}, Elrifai and Morton \cite{ElrifaiMorton}, Xu \cite{Xu}, Kang, Ko, and Lee \cite{KangKoLee}, and Birman, Ko and Lee \cite{Birman}. 
In \cite{Birman} they showed that 
for any braid $\beta \in B_n$, there exist unique $r\in\mathbb Z$, unique $k\in \mathbb Z_{\geq 0}$ and unique canonical factors $A_1,\dots,A_k$ such that 
\begin{itemize}
\item
    $\beta=\delta^rA_1 A_2\cdots A_k$ as braid elements in $B_n$ and
\item
every consecutive word $A_iA_{i+1}$ is maximally left weighted. 
\end{itemize}
The unique factorization $\delta^rA_1 A_2\cdots A_k$ of $\beta$ is called the {\em left-canonical form} of $\beta$ and is denoted by $\LCF(\beta)$. It implies that $W=W'$ in $B_n$ if and only if $\LCF(W)=\LCF(W')$.

The unique integers $r$ and $r+k$ in $\LCF(\beta)$ are called $\inf(\beta)$  and $\sup(\beta)$,  respectively. Let $\s([\beta])$ and $\SSS([\beta])$ denote the summit set and the super summit set, respectively of the conjugacy class $[\beta]$ of $\beta$.
Both sets are finite. 

Here are our characterizations of SQP and ASQP braids. 

\begin{lemma}[Lemma~\ref{cor:sqpconj}, Corollary~\ref{cor:SQP problem}]
Let $n\geq 3$. 
A braid $\beta\in B_n$ is conjugate to a SQP braid,
if and only if every element $\beta' \in \SSS([\beta])$ has $\inf(\beta')\geq 0$, 
if and only if there exists an element $\beta' \in \SSS([\beta])$ with $\inf(\beta')\geq 0$.

Thus, given an $n$-braid one can determine whether it is conjugate to a SQP braid in polynomial-time.
\end{lemma}

\begin{theorem}[Theorem~\ref{theorem:summit}, Corollary~\ref{cor:ASQP problem}]
A braid $\beta \in B_n$ is conjugate to an ASQP braid but not conjugate to a SQP braid if and only if these exists $\beta' \in \s([\beta])$ such that $\inf(\beta')=-1$ and its left-canonical form $\LCF(\beta')\equiv\delta^{-1}A_1\ldots A_{k}$ satisfies $||A_1||=n-2$.

Thus, given an $n$-braid one can determine whether it is conjugate to an ASQP braid but not to a SQP braid in exponential-time.
\end{theorem}

It is noteworthy that this ASQP problem is the first example where the summit set $\s([\beta])$ plays an essential role in computational braid theory.
 In contrast, the super summit set $\SSS([\beta])$ and its refinements like the ultra summit set $\USS([\beta])$ \cite{Gebhardt} or sliding circuits ${\tt SC}([\beta])$ \cite{GG-sliding-circuit} can be used to solve other algebraic problems in braid groups, such as, the Conjugacy problem, to compute centralizers \cite{FG-centralizer}, to find $n$th roots \cite{Lee} and possibly more. Furthermore, we can use super summit sets and its refinements based on Birman-Ko-Lee generators for geometric problems such as to determine the the Nielsen-Thurston type \cite{Ca}, or, to solve the SQP problem.

It is not known whether the super summit set or its refinements can be used to solve the ASQP problem. 
However,
for $3$- and $4$-braids, ASQP braids can be detected by using the super summit set, which is much smaller than the summit set:

\begin{theorem}[Theorem~\ref{thm:stronger-version}]\label{thm ASQP-intro}
A braid $\beta \in B_3$ is conjugate to an ASQP braid but not to a SQP braid if and only if 
every $\beta' \in \SSS([\beta])$ has $\inf(\beta')=-1$.

A braid $\beta \in B_4$ is conjugate to an ASQP braid but not to a SQP braid if and only if 
every $\beta' \in \SSS([\beta])$ has $\inf(\beta')=-1$ and its left-canonical form $\LCF(\beta')\equiv\delta^{-1}A_1\ldots A_{k}$ contains a canonical factor  $A_i$ with $||A_i||=2$. 
\end{theorem}

The if direction of Theorem~\ref{thm ASQP-intro} is true for all $n\geq 5$ (Proposition~\ref{cor:SSS-ASQP1}), but the only-if-direction does not hold in general as discussed in Remark~\ref{remark:wrong}. 
In stead we conjecture the following: 

\begin{conjecture}\label{conj:ASQP}
Let $n \geq 5$. If a braid $\beta\in B_n$ is conjugate to an ASQP braid but not conjugate to a SQP braid then there exists a $\beta' \in \SSS(\beta)$ such that $\inf(\beta')=-1$ and its left-canonical form $\LCF(\beta')\equiv\delta^{-1}A_1\ldots A_{k}$ contains a canonical factor  $A_i$ with $||A_i||=n-2$. 
\end{conjecture}

Using $\LCF(\beta)$ we can also obtain upper bounds of the negative band number $\nb(\beta)$. 
\begin{theorem}[Theorem \ref{thm:3braid-nb}] 
Let $\beta\in B_n$. 
\begin{itemize}
\item
If $\inf(\beta)<0 < \sup(\beta)$ 
then
$\nb(\beta)\leq (n-2) |\inf(\beta)|.$
\item
If $\inf(\beta)<\sup(\beta)\leq 0$ 
then
$$\nb(\beta)= -{\tt writhe}(\beta) \leq (n-2) |\inf(\beta)| + |\sup(\beta)|.$$
\end{itemize}
Moreover, when $n=3$ the equality holds for both cases. 
\end{theorem}

The negative band number has deep connections to low-dimensional topology and contact geometry. Bennequin \cite{Bennequin} found a correspondence between closed braids in $\mathbb R^3$ and transverse links in the standard tight contact structure on $\mathbb R^3$. Let $\SL(K)=\max\left\{ \sel(\beta) \mid \beta\in K \right\}$ 
the maximum self-linking number of braid representatives of a knot/link type $K$ (we write $\beta\in K$), and let $\chi(K)= 2-2g(K) - |K|$ the maximal Euler characteristic. 
Here, $\sel(\beta)$ is the self-linking number of the transverse link $\widehat\beta$, $g(K)$ is the genus of the link, and $|K|$ is the number of link components. The Bennequin inequality states
\begin{equation}\label{eq:Bennequin-inequality}
\SL(K)\leq -\chi(K).
\end{equation}

This inequality largely viewed as the starting point of modern contact geometry led him to a seminal result \cite{Bennequin} which distinguishes the tight contact structure and an overtwisted contact structure on $S^3$.  
Ito and Kawamuro \cite[Definition 1.1]{IK19}  defined the {\em defect} $\D(K)$ of the Bennequin inequality by
\begin{equation}\label{eq:Defect-Def}
\D(K):=\tfrac{1}{2}(-\chi(K)-\SL(K)).
\end{equation}
They showed that $\D(K) \in \mathbb Z$ and the defect gives a lower bound for the negative band number
\begin{equation}\label{eq:inequality}
\D(K) \leq \nb(K).
\end{equation}
In fact, in \cite[Conjecture 1]{IK19} it is conjectured that 
$\D(K) = \nb(K)$. 

We switch a gear and study when a braid representative $\beta\in K$ of a link type $K$ realizes the negative band number; namely $\nb(\beta)=\nb(K)$. Given a braid $\beta\in B_n$ searching a representing braid word $W$ such that $W$ is minimal in word length (in band generators) is called {\em the shortest word problem}. The length,  $|| \beta ||$, of such a shortest word provides a lower bound on the Euler characteristic of the link type $K$;
$$i(\beta) - ||\beta|| \leq \chi(K)$$
where $i(\beta)=n$ the braid index of $\beta$.
Moreover, due to Birman and Finkelstein \cite{Birman-Finkelstein} we have
$$
\max\left\{i(\beta) - ||\beta|| \mid \  \beta \in K \right\} = \chi(K). 
$$
We first observe the following:  
\begin{theorem}[Theorem~\ref{lem:nb(beta)}]
Let $W$ be a word representing a braid $\beta$. 
\begin{enumerate}[(a)]
\item
$W$ realizes $\nb(\beta)$ if and only if $W$ gives a shortest word for $\beta$.
\item
$W$ realizes $\nb([\beta])$ if and only if $W$ gives a shortest word for $[\beta]$.
\end{enumerate}
\end{theorem} 

Thanks to the polynomial-time algorithm for shortest words for $4$-braids by Kang, Ko and Lee \cite{KangKoLee}, the above theorem implies that $\nb(\beta)$ %and $\nb([\beta])$ 
can be computed in polynomial-time for any $4$-braids. 

We also study how $\nb(\beta)$ behaves under stabilization, a fundamental braid operation.  
In \cite{Orevkov} 
Orevkov asks: 
{\em Given a quasipositive link, is any braid representing the link with minimal braid index quasi-positive?} 
This question has been (partially) solved by Hayden \cite{Hayden}. 
Here is a related question attributed to Rudolph:  
{\em Given a strongly quasipositive link, is any braid representing the link with minimal braid index strongly quasi-positive?} 
Knowing that $\nb(K)=0$ characterizes SQP links, it is natural to extend the questions to the following: 
{\em Given a link type $K$, does any braid representing $K$ with minimal braid index achieve the negative band number $\nb(K)$?} 
We answer the question negatively: 

\begin{theorem}[Theorem~\ref{prop:stabilization example}]
There exists a sequence of link types $\{ K_n \}_{n\geq 1}$ 
whose minimal braid representative $\beta_n$ has 
$\nb(\beta_n)-\nb(K_n)=n$. 
\end{theorem}
This result shows that the number of stabilizations required to reach a braid representative that achieves the negative band number can be arbitrarily large. 

\begin{acknowledgements}
The authors would like to thank Juan Gonzalez-Meneses for useful conversations during a workshop held at ICERM Brown University in April 2022. 
MCS was partially supported by LLMJ fellowship and NSF DMS-2038103. 
TI was partially supported by JSPS
KAKENHI Grant Numbers 21H04428 and 23K03110.
KK was partially supported by NSF DMS2005450. 
RS was partially supported by NSF DMS-2038103.
\end{acknowledgements}

\section{Preliminaries}\label{Sec2}

In this section we review basic materials in \cite{KangKoLee, Birman} including band generators of the braid group $B_n$, the left-canonical form, the super summit set, and the reduction operation. 
For diagrammatic expressions, the reader can find detailed discussions and examples in the author's paper \cite{DiagrammaticKeiko}.

\subsection{Band generators}

For each $i,j$ with $1\leq i<j\leq n$, define an element $a_{ij}$ of $B_n$ by
\begin{equation*}
    a_{ij}=(\sigma_{j-1}\sigma_{j-2}\cdots\sigma_{i+1})\sigma_i(\sigma_{j-1}\cdots\sigma_{j-2}\sigma_{i+1})^{-1}.
\end{equation*}
Note that $\sigma_i=a_{i, i+1}$. 
In the braid diagram, $a_{ij}$ swaps the $i$th and $j$th strands, while leaving all other strands fixed. In addition, the two strands being interchanged must lie in front of all fixed strands. 
As this can be thought as attaching a twisted band to the $i$th and $j$th strands, $a_{i,j}$ is called a {\em band generator} of $B_n$. The band generator presentation (or Birman-Ko-Lee presentation) \cite{Birman} is as follows:
\begin{equation*}
B_n= \GenRels*{ a_{i,j} \st
\begin{medsize}
\begin{aligned} & a_{j,k}a_{i,j}=a_{i,j}a_{i,k}=a_{i,k}a_{j,k},\enspace i<j<k \\[-0.5ex] %
 & a_{i,j}a_{k,l}=a_{k,l}a_{i,j}, \enspace i<j<k<l%
\end{aligned}
\end{medsize}} 
\end{equation*}

The set of band generators of $B_n$ is denoted by $\BG(B_n)$. 
Band generators $a_{i,j}$ are often called {\em positive bands} as each creates a positive crossing in the braid diagram, versus their inverses $a_{i, j}^{-1}$ are called {\em negative bands}.

In $B_4$ there are six band generators 
$$
a_1:=a_{1,2}, \,
a_2:=a_{2,3}, \,
a_3:=a_{3,4}, \,
a_4:=a_{1,4}, \,
b_1:=a_{1,3}, \,
b_2:=a_{2,4}.
$$
($a_i$ and $b_j$ are notations in \cite{KangKoLee}), see also Figure~\ref{fig:band-generators}:
\begin{figure}[h]
\centering
\includegraphics[width=6cm]{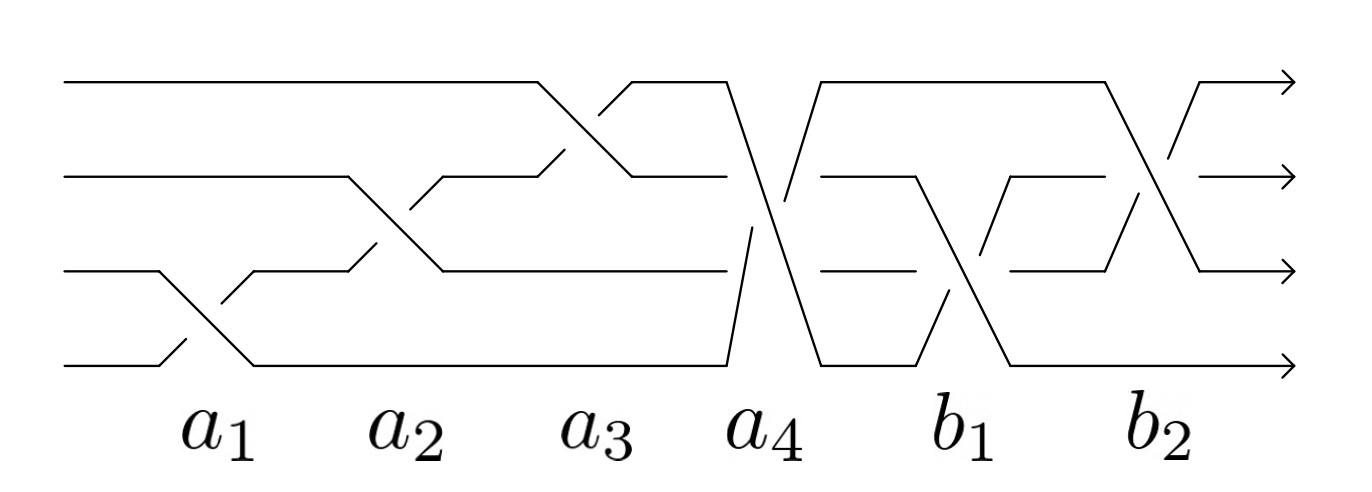}
\caption{Band generators for $B_4$. The braid strands are numbered from 1 (bottom strand) to 4 (top strand).}
\label{fig:band-generators}
\end{figure}

\subsection{Canonical factors}
In this section, we define canonical factors that play an essential role in this paper. 
We start with a partial ordering $\leq$ on the braid group $B_n$. 
\begin{definition}
An $n$-braid is {\em positive} if it is represented by a word in positive band generators. 
Let $SQP(n):=\mathcal F_0$ denote the monoid of positive $n$-braids. 

For two braid words $V,W\in B_n$, we denote $V\leq W$ if $W=PVQ$ for some positive words $P,Q \in SQP(n)$. 
\end{definition}

We introduce three invariants that will appear in the left-canonical form. 
Let $\delta:= \sigma_{n-1} \cdots \sigma_2 \sigma_1$. 
\begin{definition}\label{def:inf-sup}
For $\beta\in B_n$, define
\begin{align*}
    \inf (\beta)&:=\max\{r\in\mathbb{Z}\mid\delta^r\leq \beta\},\\
    \sup(\beta)&:=\min\{s\in\mathbb{Z}\mid \beta\leq\delta^s\},\\
    \ell(\beta)&:=\sup(\beta)-\inf(\beta).
\end{align*}
Here, $\ell(\beta)$ is called the \textit{canonical length} which is different from  $||\beta||$, the usual minimal word length in band generators. 
\end{definition}

\begin{definition}\label{def:CF}
A braid $\beta \in B_n$ is called a {\em canonical factor} if $e \leq \beta \leq \delta$. 
The set of canonical factors is denoted by $\CFn$. 
\end{definition}

For example, $a_2a_1 \in \CFn$ since $a_3a_2 a_1 = \delta$ therefore $e \leq a_2 a_1 \leq \delta$. In terms of canonical length and word length: $\ell(a_2 a_1)=1$ and $||a_2 a_1||=2.$

In \cite[Corollary 3.5]{Birman}, Birman, Ko, and Lee showed that the cardinality of the set $\CFn$ is the $n$th Catalan number $\mathcal{C}_n=(2n)!/n!(n+1)!$. In their proof, the following theorem is implicit.

\begin{theorem}\label{thm:1-1correspondence} 
The set $\CFn$ is in a one-to-one correspondence with the set of noncrossing partitions of $n$ elements. 
\end{theorem}

\begin{example}
There are 14 canonical factors in $B_4$.  See \cite{DiagrammaticKeiko} for a diagrammatic reformulation of the canonical factors. 
\begin{align*}
\CF&=\{
e, a_1, a_2, a_3, a_4, b_1, b_2, a_2a_1, a_3a_2, a_4a_3, a_1a_4, a_1a_3, a_2a_4, \delta\}\\
&=\{(\e),(\aone),(\atwo),(\athree),(\afour),(\bone),(\btwo),\\&\enspace\enspace\enspace\enspace(\atwoaone),(\athreeatwo),(\afourathree),(\aoneafour),
    (\aoneathree),(\atwoafour),(\delt)\}.
\end{align*}
\end{example}

Let $\tau:B_n\rightarrow B_n$ be an inner automorphism defined by 
$$\tau(\beta)=\delta^{-1}\beta \delta.$$ 
Since $\tau(\sigma_i)=\sigma_{i+1}$, 
diagrammatically $\tau$ rotates each diagram $2\pi/n$ counterclockwise.
For example, $\tau(a_i)= a_{i+1} $ and $\tau(b_i)= b_{i+1}$. 
In particular, $\tau$ preserves the set $\CFn$.

\begin{theorem}\label{cor_prec}
The following are equivalent. Let $A, B \in \CFn$. 
\begin{itemize}
\item
$PAQ=B$ for some $P, Q \in \CFn$, i.e., $A \leq B$.
\item 
$AQ=B$ for some $Q\in\CFn$, which we denote $A\prec B$. 
\item
$PA=B$ for some $P \in \CFn$. 
\end{itemize}
\end{theorem}

\begin{corollary}\label{cor:complementary} \cite{DiagrammaticKeiko}
For every canonical factor $A\in \CFn$ there exist $A'$ and $A'' \in\CFn$ such that $A' A=A A''=\delta.$ 
\end{corollary}

Figure~\ref{fig:HasseDiagram} describes the partially ordered set $(\CF, \prec)$. 
\begin{figure}[h]
\centering
\includegraphics[width=8cm]{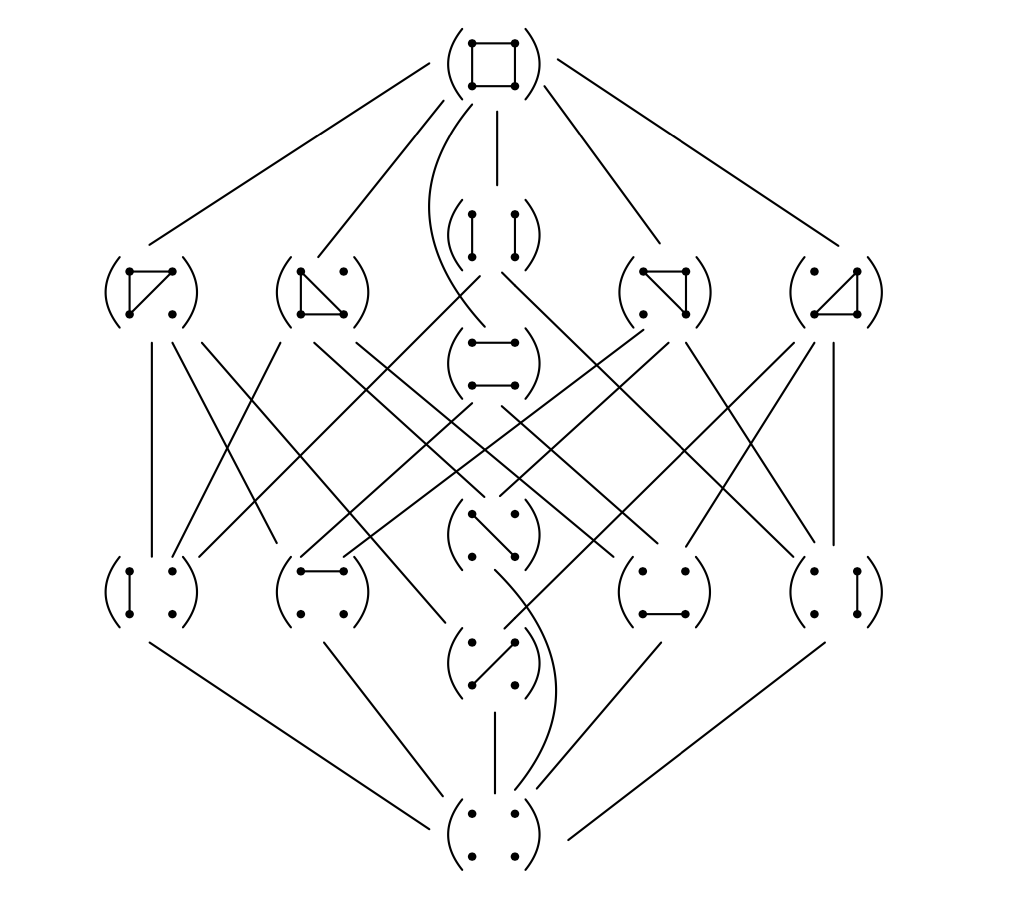}
\caption{The partially ordered set $(\CF, \prec)$. The line connecting a lower diagram $A$ and an upper diagram $B$ means $A \prec B$.} 
\label{fig:HasseDiagram}
\end{figure}

\begin{definition}\label{def:Rightarrow}
Let $A,B,A',B'\in\CFn$. We denote
$$        AB\Rightarrow A'B' $$
if the following are satisfied:
    \begin{enumerate}[a)]
        \item $AB=A'B'$ as braid elements in $B_n$.
        \item $A \prec A'$. 
    \end{enumerate}
In other words, the word $A'B'$ is more {\em left-weighted} than $AB$. 
In case of $B'=e$, $B'$  may be omitted.
\end{definition}

\begin{example}\label{ex:1}
\begin{enumerate}
\item
If $A \in \CFn$ then $A \enspace\delta \Rightarrow \delta\enspace\tau(A).$
\item
$(\bone)(\atwoafour)\Rightarrow (\afourathree)(\atwo)\Rightarrow (\delt) (\e)= (\delt).$
\item
$(\aoneathree)(\bone)
\Rightarrow (\aoneathree)\diamond(\bone)  = (\delt).$
\end{enumerate}
\end{example}

We define maximally left-weighted words that play an important role in the definition of the left-canonical form:

\begin{definition}
Let $A, B \in \CFn$. We say $AB$ is {\em maximally left-weighted} if no non-trivial $A', B' \in \CFn$ satisfy $AB\Rightarrow A'B'$.
\end{definition}

\begin{example}\label{ex:max left-weight}
The word $(\aoneathree)(\btwo)$ is maximally left-weighted.
\end{example}

\subsection{The left-canonical form}\label{sec:LCF}
\begin{theorem}\label{thm:LCF}
\cite{Xu, KangKoLee, Birman}

For any $n$-braid $\beta$, there exist unique $r\in\mathbb Z$, unique $k\in \mathbb Z_{\geq 0}$ and unique canonical factors $A_1,\dots,A_k\in\CFn \setminus \{e, \delta\}$ such that 
\begin{itemize}
\item
    $\beta=\delta^rA_1 A_2\cdots A_k$ as braid elements in $B_n$ and
\item
any consecutive pairs 
$A_1 A_2, \cdots, A_{k-1}A_k$ are maximally left weighted. 
\end{itemize}
Moreover
$$\inf(\beta)=r, \quad \sup(\beta)=r+k, \ \mbox{ and } \ \ell(\beta)=k.$$
The unique factorization $\delta^rA_1 A_2\cdots A_k$ of $\beta$ is called the {\em left-canonical form} of $\beta$ and denoted by $\LCF(\beta)$. 
\end{theorem}

See \cite{KangKoLee} for the details of Kang, Ko, and Lee's algorithm for the left-canonical form for 4-braids, which has been extended to general $n$-braids \cite{Birman} as follows. An important observation is that this is a polynomial-time algorithm.

Suppose 
$\beta=A_1 A_2 \cdots A_k \in SQP(n)$ where $A_i \in \BG(B_4) \subset \CF$. 
Change the factorization of $\beta$ by the following rules.
\begin{enumerate}
\item[(i)] If $A_i=e$ for some $i$, then remove $A_i$ from the word $A_1\cdots A_k$. 
\item[(ii)] Replace $A_i A_{i+1}$ with $A'_i A'_{i+1}$ so that $A_i A_{i+1} \Rightarrow A'_i A'_{i+1}$. 
\end{enumerate}
Apply (i) and (ii) until the partial order cannot be raised anymore. The resulting word is the left-canonical form, $\LCF(\beta)$, as defined in Theorem~\ref{thm:LCF} and it does not depend on the way we apply (a) and (b).

Suppose that $\beta\notin SQP(n)$.
Let $\beta=A_1^{\epsilon_1}A_2^{\epsilon_2}\cdots A_k^{\epsilon_k}$, where $A_i \in \CF$, 
$\epsilon_i=\pm 1$, and 
$\epsilon_i=-1$ for some $i$. 
We will get rid of the negative exponent terms as follows. 
By Corollary~\ref{cor:complementary} there exists a canonical factor $A_i' \in \CF$ such that $A_i A_i' = \delta$. 
We replace $A_i^{-1}$ with $A_i'\delta^{-1}$. 
Applying the rule $A\delta^{-1}=\delta^{-1}\tau^{-1}(A)$ for any $A\in\CFn$ a number of times will shift $\delta^{-1}$ to the beginning of the word. 
Repeat this procedure until we exhaust all the factors with negative exponents. It will give us a word $\delta^{-r}P$ for some $r \in \mathbb N$ and positive word $P \in SQP(n)$. Finally we  apply the above algorithm to the positive word $P$ to obtain $\LCF(\beta$).

\subsection{Summit set and super summit set} 

For later use, we review some properties of summit sets and super summit sets.

\begin{definition}\label{def:inf[beta]}
For the conjugacy class $[\beta]$ of $\beta \in B_n$ we define: 
\begin{eqnarray*}
\inf([\beta])&:=& \max\{ \inf(\beta') \mid \beta' \mbox{ is conjugate to } \beta\} \\
\sup([\beta]) &:=& \min\{ \sup(\beta') \mid \beta' \mbox{ is conjugate to } \beta\}
\end{eqnarray*}
\end{definition}

\begin{definition}
The \emph{summit set} of a braid $\beta$, denoted by $\s([\beta])$, is the set of conjugates $\beta'$ of $\beta$ such that its infimum $\inf(\beta')$ is maximal among all conjugates of $\beta$. That is, $\inf(\beta')=\inf([\beta])$.
\end{definition}

The summit set appears in Garside's original solution of the conjugacy problem \cite{Garside}. 
It has served as a prototype for the later developments of more effective solutions of the conjugacy problem for braid groups and Garside groups.

\begin{definition}\label{def:SSS}\cite{Birman, ElrifaiMorton, Garside} 
The {\em super summit set} of a braid $\beta$, denoted $\SSS([\beta])$, is the set of conjugates $\beta'$ of $\beta$ such that the canonical length
$\ell(\beta')$ is minimal among all the conjugates of $\beta$. 
\end{definition}

In \cite[Corollary 4.6]{KangKoLee} it is proved that the $\inf[\beta]$ and $\sup[\beta]$ can be achieved simultaneously by some $\beta' \in [\beta]$. Since the canonical length $\ell(\beta)=\sup(\beta)-\inf(\beta)$ we have the following:

\begin{proposition}\label{prop:sss}
Any super summit element $\beta' \in \SSS([\beta])$ realizes $\inf([\beta])$ and $\sup([\beta])$. Namely $\inf(\beta')=\inf([\beta])$ and $\sup(\beta')=\sup([\beta])$. 
\end{proposition}

An important fact proved by Elrifai and Morton is that: 
two braids are conjugate if and only if their super summit sets are identical if and only if their super summit sets intersect \cite{ElrifaiMorton,Garside}. 
The conjugacy problem in terms of band generators has been solved for $B_3$ by Xu \cite{Xu}, for $B_4$ by Kang, Ko, and Lee \cite{KangKoLee} applying the technique of Elrifai and Morton, and for $B_n$ by Birman, Ko, and Lee \cite{Birman}. Their algorithms use the $\LCF(\beta)$ and explicitly compute the $\SSS([\beta])$.
Although computing the whole $\SSS([\beta])$ requires exponential-time in general, finding one element of $\s([\beta])$ or $\SSS([\beta])$, (and hence computing $\inf([\beta]), \sup([\beta])$ and $\ell([\beta])$) can be done in polynomial-time \cite{Birman}.

\subsection{Reduction and the shortest word problem} 
\begin{definition}\label{def:Red}
Given a word  $W=\delta^rW_1W_2\cdots W_s$ with $W_i\neq e$ and 
$W_i$ or $W_i^{-1} \in \CFn \setminus \{\delta, e\}$ 
define $\red(W)$ as follows:
\begin{enumerate}[(1)]
    \item if $r\geq 0$ or $W_i<e$ for all $i=1,\dots,s$, then $\red(W)=W$;
    \item otherwise (i.e., $r<0$ and there exists some $i$ with $W_i\in\CF$), choose a word $W_k\in \CFn$ whose word length $||W_k||$ is maximal among all $W_1, \dots, W_k$. 
By Corollary~\ref{cor:complementary} there exists $V_k \in \CF$ such that $W_k \diamond V_k = W_k V_k = \delta$. Put $W_k' = (V_k)^{-1}.$
We define
    \begin{equation*}
        \red(W)=\delta^{r+1}\tau(W_1)\cdots\tau(W_{k-1})W_k'W_{k+1}\cdots W_s,
    \end{equation*}
\end{enumerate}
Define $\Red(W):=\red^{|r|}(W)$, i.e., repeat the above algorithm until the exponent of $\delta$ is non-negative. We note that both $\red(W)$ and $\Red(W)$ depend on choices. 
\end{definition}

\begin{example}
Let 
    \begin{align*}
        W&=\delta^{-2}(a_3a_2)(a_4a_3)a_4b_1b_2
        =(\delt)^{-2}(\athreeatwo)(\afourathree)(\afour)(\bone)(\btwo)\\
        &=\delta^{-2}W_1W_2W_3W_4W_5.
    \end{align*}
    Here, $r<0$, so we proceed to option (2) in the algorithm. We choose the word $W_2=(\afourathree)=(a_4a_3)$ of maximal word length among all $W_i$. 
Since $(\afourathree)\diamond(\atwo)=(\delt)$ we set $W_2'=(\atwo)=a_2$.      
Thus,
$$
        \red(W)=\delta^{-1}\tau(W_1)W_2'W_3W_4W_5
=\delta^{-1}(a_4a_3)a_2^{-1}a_4b_1b_2.
$$
    We repeat this process one more time. We choose the word $\tau(W_1)=a_4a_3$ of maximal length among all $W_i$. Then,
    \begin{equation*}
        \Red(W)=\red^2(W)=a_2^{-1}a_2^{-1}a_4b_1b_2.
    \end{equation*}
\end{example}

Here is Kang, Ko and Lee's solution to the shortest word problem for a 4-braid and also for the conjugacy class of a $4$-braid.
Both solutions come with a polynomial-time algorithm since computing one element of $\SSS([\beta])$ can be done in polynomial-time.

\begin{theorem}\label{lem:KKL-reduced} 
\cite[Lemma 5.1]{KangKoLee}
Let $n\leq 4$ and $\beta\in B_n$. For any word representative $W$ of $\beta$, every reduced word of the left-canonical form minimizes the word length of $\beta$; namely, 
$$||\Red(\LCF(\beta))|| \leq ||W||.$$

Similarly, for every $\beta' \in \SSS([\beta])$, $||\Red(\LCF(\beta'))||$ minimizes the word length of braids conjugate to $\beta$.
\end{theorem}

\section{Detection of SQP and ASQP braids}\label{Sec:6}

In this section, we characterize SQP braids and ASQP braids in terms of the left-canonical form. 
We assume $n\geq 3$.

\begin{lemma}\label{Thm:SQ}
An $n$-braid $\beta\in B_n$ is SQP if and only if $\inf(\beta)\geq 0$.
\end{lemma}
\begin{proof} 
($\Leftarrow$) Trivial. 
($\Rightarrow$) Let $\beta\in B_n$ be strongly quasipositive. By definition, $\beta$ can be represented by a positive word, $W$, in band generators of $B_n$. Applying the LCF-algorithm  (Section~\ref{sec:LCF}) to $W$ we obtain $\inf(\beta)\geq 0$. 
\end{proof}

\begin{lemma}\label{cor:sqpconj}
The following are equivalent:
\begin{itemize}
\item A braid $\beta\in B_n$ is conjugate to a SQP braid.
\item Every element $\beta' \in \SSS([\beta])$ has $\inf(\beta')\geq 0$.
\item There exists an element $\beta' \in \SSS([\beta])$ with $\inf(\beta')\geq 0$.
\end{itemize}
\end{lemma}

\begin{proof}
The statement follows from Proposition~\ref{prop:sss} and Lemma~\ref{Thm:SQ}. 
\end{proof}

\begin{corollary}[SQP problem]\label{cor:SQP problem}
The problem to determine whether a given braid is conjugate to a SQP braid can be solved in polynomial-time with respect to both the word length and the number of braid strands. 
\end{corollary}

Here is a sufficient condition for an ASQP braid: 

\begin{proposition}\label{cor:SSS-ASQP1}
A braid $\beta\in B_n$ is conjugate to an ASQP braid but not conjugate to any SQP braid if there exists a $\beta' \in \SSS(\beta)$ such that $\inf(\beta')=-1$ and at least one canonical factor  $A_i$ in $\LCF(\beta')\equiv\delta^{-1} A_1 \cdots A_k$ has length $||A_i||=n-2$. 
\end{proposition}

\begin{proof}
Suppose that $\beta' \in \SSS(\beta)$ has $\inf(\beta')=-1$ and $\LCF(\beta')\equiv\delta^{-1}A_0\cdots A_k$ contains an $A_i$ with length $||A_i||=n-2$. 
Then by Lemma~\ref{cor:sqpconj} $\beta$ cannot be conjugate to a SQP braid. 

Corollary~\ref{cor:complementary}  implies that there is a unique $C_i \in \CFn$ such that 
$A_i C_i = \delta$. %$\delta^{-1}A_i=C_i^{-1}$. 
Since $||\delta^{-1}A_i||=(n-1)-(n-2)=1$ we have $||C_i||=1$. 
Applying the relation $\delta^{-1} A \delta =\tau(A)$ for $A\in\CFn$ to the $\LCF(\beta')$ we obtain
$$\beta' = \tau(A_0) \cdots \tau(A_{i-1})C_i^{-1} A_{i+1}\cdots A_k.$$ 
This shows $\beta'$ is ASQP since $\tau(A)\in\CFn$. 
\end{proof}

\begin{theorem}
\label{theorem:summit}
A braid $\beta \in B_n$ is conjugate to an ASQP braid but not conjugate to a SQP braid if and only if these exists $\beta' \in \s([\beta])$ such that $\inf(\beta')=-1$ and its left-canonical form $\LCF(\beta')\equiv\delta^{-1}A_1\ldots A_{k}$ satisfies $||A_1||=n-2$.
\end{theorem}
\begin{proof}
If part follows from Proposition~\ref{cor:SSS-ASQP1}.  
To see the only if part, assume $\beta$ is conjugate to an ASQP braid $\beta'$. Without loss of generality we may assume that $\beta'=a^{-1}A_2\cdots A_k$  for some band generator $a \in \BG(B_n)$ and $A_2, \dots, A_k \in \CFn$ where $\LCF(a \beta')\equiv A_2\cdots A_k$.  
Let $A_1 = \delta a^{-1}$ so that $\beta'=\delta^{-1}A_1\cdots A_k$. 
Note that $A_1$ is a canonical factor with $||A_1||=n-2$.

We claim that $\delta^{-1}A_1\cdots A_k$ is the left-canonical form of $\beta'$. 
If not, $A_1A_2$ must not be maximally left-weighted. 
Since $||A_1||=n-2$ this means that $\LCF(A_1 A_2)\equiv\delta A'$ for some $A' \in \CFn$. It implies that $\beta' = A'A_3\cdots A_k$ is a SQP braid, a contradiction.
\end{proof}

Although the summit set $\s([\beta])$ is much larger than the super summit set $\SSS([\beta])$, we know $\s([\beta])$ is still a finite set and one can compute it almost the same way as the super summit set via repeated cycling.
Theorem \ref{theorem:summit} implies the following exponential-time solution to the ASQP problem:

\begin{corollary}[ASQP problem, exponential-time]\label{cor:ASQP problem}
The problem to determine whether a given braid is conjugate to an ASQP braid but is not conjugate to any SQP braid can be solved in exponential-time with respect to both the word length and the number of braid strands.
\end{corollary}

For the $4$-braid case the super summit set can be used to solve the ASQP problem and we have a stronger version of Theorem~\ref{theorem:summit}.

\begin{theorem}\label{thm:stronger-version}
Let $\beta \in B_4$. The following are equivalent:
\begin{enumerate}
\item
A braid $\beta$ is conjugate to an ASQP braid but not conjugate to any SQP braid.
\item
Every $\beta' \in \SSS([\beta])$ has $\inf(\beta')=-1$ and its left-canonical form $\LCF(\beta')\equiv\delta^{-1}A_1\ldots A_{k}$ contains a canonical factor  $A_i$ with $||A_i||=2$. 
\item
There exists a $\beta' \in \SSS([\beta])$ with $\inf(\beta')=-1$ and its left-canonical form $\LCF(\beta')\equiv\delta^{-1}A_1\ldots A_{k}$ contains a canonical factor  $A_i$ with $||A_i||=2$. 
\end{enumerate}
\end{theorem}

\begin{proof}
(2)$\Rightarrow$(3) is trivial.
(3)$\Rightarrow$(1) follows by Proposition~\ref{cor:SSS-ASQP1}. 
We prove (1)$\Rightarrow$(2). 
Assume, to the contrary that there is $\beta' \in \SSS([\beta])$ such that every canonical factor $A_i$ in $\LCF(\beta')\equiv\delta^{-1}A_1\ldots A_{k}$ has $||A_i||=1$. 
Then the exponent sum $e(\beta')=e(\beta) = -3+k$.
Hence the canonical length $\ell(\beta')=k= 3+e(\beta)$. 
On the other hand, we already know by Theorem~\ref{theorem:summit} that there is $\beta'' \in \s([\beta])$ such that $\LCF(\beta'')= \delta^{-1}A'_1\ldots A'_{k'}$ with $||A'_1||=2$. Thus
$-3+2+(k'-1)\leq e(\beta'') = e(\beta)$ so its canonical length satisfies $\ell(\beta'')=k' < 2+e(\beta) < k=
\ell(\beta')$. This contradicts that $\beta' \in \SSS([\beta])$, the minimality of its canonical length.
\end{proof}

This leads to the following polynomial-time solution to the ASQP problem for $3$- and $4$-braids:

\begin{corollary}[ASQP problem, polynomial-time]
For $3$- and $4$-braids, the problem whether given braid is conjugate to an ASQP braid but is not conjugate to any SQP braid can be solved in polynomial-time with respect to the word length.
\end{corollary}

\begin{proof}
Let $\beta \in B_4.$ 
Let $\alpha \in [\beta]$. Applying cycling and decycling to $\LCF(\alpha)$ will give an element, say $\beta'$, of $\SSS([\beta])$ in polynomial-time. 
Theorem~\ref{thm:stronger-version} states that 
$\beta'$ satisfies the conditions in (3), which can be verified in polynomial-time, if and only if $\beta$ is conjugate to an ASQP braid but not conjugate to any SQP braid.
\end{proof}

Conjecture~\ref{conj:ASQP} claims that the converse of Proposition~\ref{cor:SSS-ASQP1} to be true for all $n\geq 5$. 

\begin{remark}\label{remark:wrong}
In Conjecture~\ref{conj:ASQP} `there exists a $\beta'$ such that' cannot be replaced with `every $\beta'$ has'.   
For example, the $5$-braids $\beta$ and $\beta'$ in %Figure~\ref{fig:exSSS} 
the figure are conjugate. 
One can verify that both are in $\SSS(\beta)$ and also in their left-canonical forms. The first canonical factor in $\LCF([\beta])$ has length $3= n-2$ but none of the canonical factors in $\LCF(\beta')$ have length $3$. 
\begin{figure}[htbp]
\begin{center}
   \includegraphics*[width=7cm, bb=110 30 410 120]{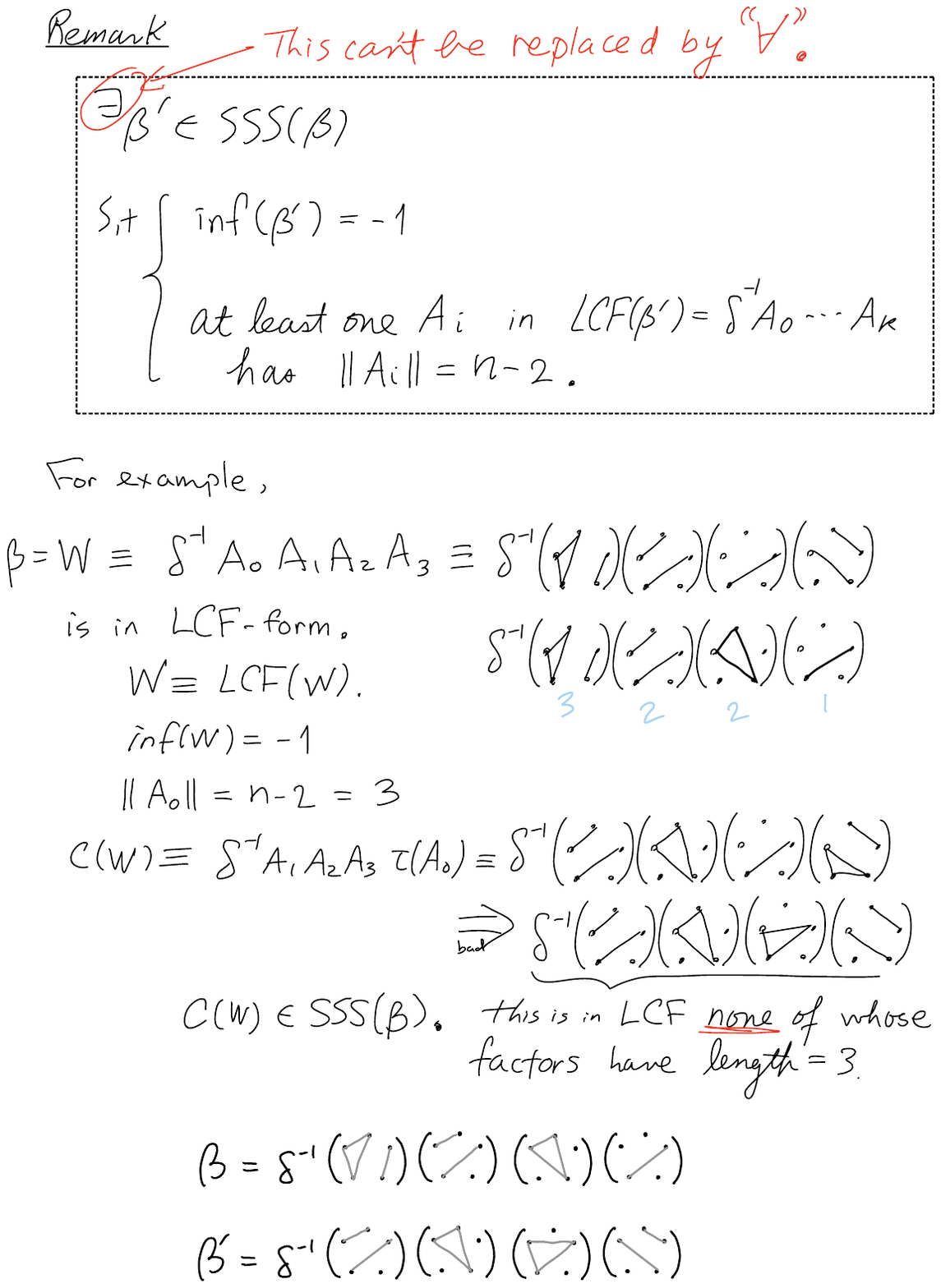} % requires the graphicx package
\end{center}
\end{figure}
It is interesting to point out that $\beta'\in\USS([\beta])$ is an element of the {\em ultra summit set} \cite{Gebhardt} of $\beta$; namely, there exists some natural number $d$ such that applying the cycling operation \cite{ElrifaiMorton} to $\beta'$ for $d$ times will give us back  $\beta'$. However, $\beta \notin \USS([\beta])$. 

\end{remark}

\section{Bounds of $\nb(\beta)$ via $\LCF(\beta)$}\label{Sec:7}

We will study the negative band numbers $\nb(\beta)$ in terms of the left-canonical form. 
A basic observation is: $\nb(\beta)\geq 1$ if and only if $\inf(\beta) \leq -1$. 
We start with a lower bound of $\nb(\beta)$. 

\begin{theorem}\label{thm:inequality}
Let $\beta\in B_n$ with $n\geq 3$ and $\inf(\beta) \leq -1$. Then 
$$
0< -\inf(\beta)= |\inf(\beta)| \leq \nb(\beta). 
$$
\end{theorem}

\begin{proof}
The first inequality follows from Lemma~\ref{Thm:SQ}. 

Suppose that $\beta$ is represented by a word $W$ in band generators with $\nb(\beta) (\geq 1)$ negative bands. Apply the LCF-algorithm to $W$. The algorithm first replaces each negative band with an $A \delta^{-1}$ for some canonical factor $A \in \CFn$, then shifting $\delta^{-1}$ to the left. During the shifting process it is possible that some $\delta^{-1}$ vanishes if cancellation occurs. 
Therefore, we obtain 
$-\nb(\beta) \leq \inf(\beta) <0$.
\end{proof}

Here is an upper bound for $\nb(\beta)$. 
\begin{theorem}\label{thm:3braid-nb}
Let $\beta\in B_n$ with $n\geq 3$ and $\inf(\beta) \leq -1$. 
\begin{itemize}
\item
If $\inf(\beta)<0 < \sup(\beta)$ 
then
$\nb(\beta)\leq (n-2) |\inf(\beta)|.$
\item
If $\inf(\beta)<\sup(\beta)\leq 0$ 
then
$$\nb(\beta)= -{\tt writhe}(\beta) \leq (n-2) |\inf(\beta)| + |\sup(\beta)|.$$
\end{itemize}
Moreover, when $n=3$ the equality holds for the both cases. 
\end{theorem}

\begin{proof}
Put $r:=-\inf(\beta)=|\inf(\beta)|>0$ and $k:=\ell(\beta)\geq 0$. 
We have $$\sup(\beta)=-r+k.$$
Suppose that 
$\LCF(\beta)=\delta^{-r} A_1 \cdots A_k$ where $A_1,\cdots,A_k \in \CFn \setminus \{e, \delta\}$.  
By Corollary~\ref{cor:complementary} for each $A_i$ there exists $A_i' \in \CFn \setminus \{e, \delta\} $ such that $A_i \diamond A_i'=A_i A_i'=\delta$.

Recall the inner automorphism $\tau:B_n \to B_n$ defined by $\tau(\beta)=\delta^{-1}\beta\delta$. For a canonical factor $A\in \CFn$, $\tau(A)$ is diagrammatically counterclockwise $2\pi/n$ rotation of $A$ and $\tau(A^{-1})=(\tau(A))^{-1}$. 
Thus $\tau$ preserves the word length; $$||\tau(A)||=||A||=||\tau(A^{-1})||.$$

{\bf Case 1:} 
Suppose $\inf(\beta)<0 < \sup(\beta)$. 
We apply the reduction operation (Definition~\ref{def:Red}-(2)) to $\LCF(\beta)$ and obtain a $\Red(\LCF(\beta))$. 
Among the canonical factors $A_1, \cdots, A_k$ in $\LCF(\beta)$ the ones with $r$ largest word length, say $A_{i_1}, \cdots, A_{i_r}$, are replaced by the negative words ${(A'_{i_1})}^{-1}, \cdots, {(A'_{i_r})}^{-1}$ up to rotation by $\tau$, and $\delta^{-r}$ disappears in $\Red(\LCF(\beta))$. 
The rest of the $k-r$ canonical factors in $\LCF(\beta)$ are kept the same up to rotation by $\tau$.

Since every canonical factor in $\CFn \setminus \{e, \delta\}$ has word length at most $n-2$ and $\tau$ preserves the word lengths of canonical factors (and their inverses), 
the number of negative bands in $\Red(\LCF(\beta))$ is at most $(n-2)r$.
This gives $$\nb(\beta) \leq (n-2)r = (n-2)|\inf(\beta)|.$$ 

When $n=3$ we can say further.
Every canonical factor in $$\Cf\setminus\{e, \delta\}=\{a_1, a_2, a_3\}$$ has word length exactly $1$. This means each of ${(A'_{i_1})}^{-1},$ $\cdots,$ ${(A'_{i_r})}^{-1}$ contributes exactly one negative band and we obtain $$\nb(\beta) = r= |\inf(\beta)|.$$

{\bf Case 2:}
If $\inf(\beta)<\sup(\beta)\leq 0$ 
then $\beta$ can be represented by a negative word in band generators. 
Thus 
\begin{eqnarray*}
\nb(\beta) &=& -{\tt writhe}(\beta)\\
&=&
||\delta^{-r}|| - (||A_1||+ \cdots + ||A_k||) \\
&\leq& 
(n-1)r - k \\
&=&
 (n-2)|\inf(\beta)| + |\sup(\beta)|.
 \end{eqnarray*}
For 3-braids, since $||\delta||=2$ and $||A_i||=1$, we get $\nb(\beta)=-{\tt writhe}(\beta)=|\inf(\beta)|+|\sup(\beta)|$. 
\end{proof}

\begin{remark}\label{rmk:by Ito}
For {\bf Case 1} of the above proof we have the following slight improvement of the upper bound: 
For $\LCF(\beta)=\delta^{-r} A_1 \cdots A_k$, the negative band number satisfies 
\begin{equation}\label{eq:ito-upperbound}
\nb(\beta) \leq (n-1)r - \frac{r}{k}\sum_{i=1}^{k} || A_i || 
\end{equation}
where $|| A_i ||$ is the word length of the canonical factor in band generators.  
Thus, $\sum_{i=1}^{k} || A_i ||= {\tt writhe}(\beta)+ (n-1) |\inf(\beta)|$. 

In the spacial case where $|| A_1 ||=\cdots=|| A_k ||=1$, the right hand side of (\ref{eq:ito-upperbound}) will be equal to $r(n-2)$ which is the same upper bound as in Theorem~\ref{thm:3braid-nb}.

The bound (\ref{eq:ito-upperbound}) is obtained as follows. 
Since $A_i \diamond A_i' = \delta$, we have $||A_i'||=||\delta||-||A_i||=(n-1)-||A_i||.$ We also recall that $A_{i_1}, \cdots, A_{i_r}$ are chosen so that 
$\frac{1}{k}(||A_1||+\cdots+||A_k||) \leq \frac{1}{r}(||A_{i_1}||+\cdots+||A_{i_r}||)$ holds. Counting the negative bands in $\Red(\beta)$ we obtain 
\begin{eqnarray*}
\nb(\beta) &\leq& \sum_{j=1}^r ((n-1)-||A_i||) \\
&\leq& r(n-1) - \frac{r}{k}(||A_1||+\cdots+||A_k||).
\end{eqnarray*}
\end{remark}

\section{Shortest words and $\nb(\beta)$}

The negative band number $\nb(\beta)$ is realized by a shortest word: 

\begin{theorem}\label{lem:nb(beta)}
Let $W$ be a word representing $\beta \in B_n$. 
\begin{enumerate}[(a)]
\item
$W$ realizes $\nb(\beta)$ if and only if $W$ gives a shortest word for $\beta$.
\item
$W$ realizes $\nb([\beta])$ if and only if $W$ gives a shortest word for $[\beta]$.
\end{enumerate}
\end{theorem}

\begin{proof}
We only show the statement (a). The statement (b) follows by a similar argument. 
Let $W$ be a shortest word representing $\beta$ and let $p$ (resp. $n$) be the number of positive (resp. negative) bands in the word $W$. 
Assume that a braid word $W'$ represents $\beta$ and realizes $\nb(\beta)$. 
Let $p'$ (resp. $n'$) be the number of positive (resp. negative) bands in the word $W'$.
We get $$n'=\nb(\beta)\leq n.$$

Since $W$ is a shortest word, 
$$
p'+n'=||W'|| \geq ||W||=p+n.
$$
Since $W$ and $W'$ belong to the same conjugacy class $[\beta]$, they have the same writhe. 
$$
p'-n'={\rm writhe}(W')={\rm writhe}(W)=p-n
$$
Therefore, $n' \geq n$. 

Above two paragraphs give $n=n'=\nb(\beta)$. 
\end{proof}

As a corollary of Theorem~\ref{lem:KKL-reduced}  and Lemma~\ref{lem:nb(beta)}, we observe the following: 

\begin{corollary}\label{cor:nb(beta)} 
Let $n\leq 4$ and $\beta\in B_n$. Every reduced word $W=\Red(\LCF(\beta))$ of the left-canonical form achieves $\nb(\beta)$.
\end{corollary}

For $n\geq 5$ the result does not hold. There exists words $\beta \in B_n$ for $n\geq5$ such that $Red(LCF(\beta))$ is not a shortest word and therefore by Theorem~\ref{lem:nb(beta)} does not achieve the negative band number $\nb(\beta)$. An algorithm for computing $\nb(\beta)$ for $n\geq 5$ will be explored in future work.

\section{Behavior of $\nb(\beta)$ under stabilization}

For a braid word $W$ let $F_W$ denote the Bennequin surface \cite{Bennequin} associated to $W$. If $W$ is an $n$-braid the surface $F_W$ is obtained by $n$ parallel disks joined by twisted bands corresponding to the word $W$. 
For a braid $\beta$ define $$\chi(F_\beta):=\max\left\{ \chi(F_W) \mid W \in \beta \right\},$$ 
where $W\in \beta$ means $W$ is a word representing the braid $\beta$. 
Thus, $\chi(F_\beta)$ is achieved by a shortest word $W$ representing $\beta$.

Immediately the Bennequin inequality (\ref{eq:Bennequin-inequality}) extends to the following:
\begin{equation}\label{eq:long-Bennequin}
\sel(W)\leq \sel(\beta) \leq \SL(K)\leq -\chi(K) \leq -\chi(F_\beta) \leq \chi(F_W).
\end{equation}
Recall the {\em defect} of the Bennequin inequality 
$\D(K):=\tfrac{1}{2}(-\chi(K)-\SL(K))$ which gives a lower bound for the negative band number
$\D(K) \leq \nb(K).$
Similarly we may define 
$\D(\beta):=\frac{1}{2}(-\chi(F_\beta)-\sel(\beta))$. Then we have: 
$$
\begin{array}{ccccc}
&& \nb(K) &&\\
&\leq & & \leq &\\
\D(K) & & & & \nb(\beta)\\
&\leq & & \leq &\\
&& \D(\beta) &&
\end{array}
$$
For $W\in \beta \in K$ the following statements are in general independent: 
\begin{itemize}
\item
$W$ realizes $\chi(K)$, i.e., $\chi(F_W)=\chi(K).$
\item
$W$ realizes $\SL(K)$, i.e., $\sel(W)=\sel(\beta)=\SL(K).$
\end{itemize}
See Remark~\ref{final-remark} for such examples. 
By (\ref{eq:long-Bennequin}) it is easy to see that:
\begin{itemize}
\item
If $W$ realizes $\D(K)$ then $W$ realizes both $\chi(K)$ and $\SL(K)$. 
\end{itemize}
Therefore, it is interesting to find a braid word that simultaneously realizes $\chi(K), \SL(K), \D(K)$ and $\nb(K)$.

We first study how $\nb(\beta)$ behaves under braid stabilization, a basic braid operation. 
For the usual $\pm$-stabilization $B_n \ni \beta \mapsto \beta \sigma_n^\pm \in B_{n+1}$ it is easy to see for many cases 
$$
\nb(\beta)=\nb(\beta \sigma_n) = \nb(\beta \sigma_n^{-1})-1.
$$
However for certain stabilizations this is not always the case. 
We give such stabilizations in the following:

\begin{figure}[h]
 \centering
\includegraphics*[width=12cm]{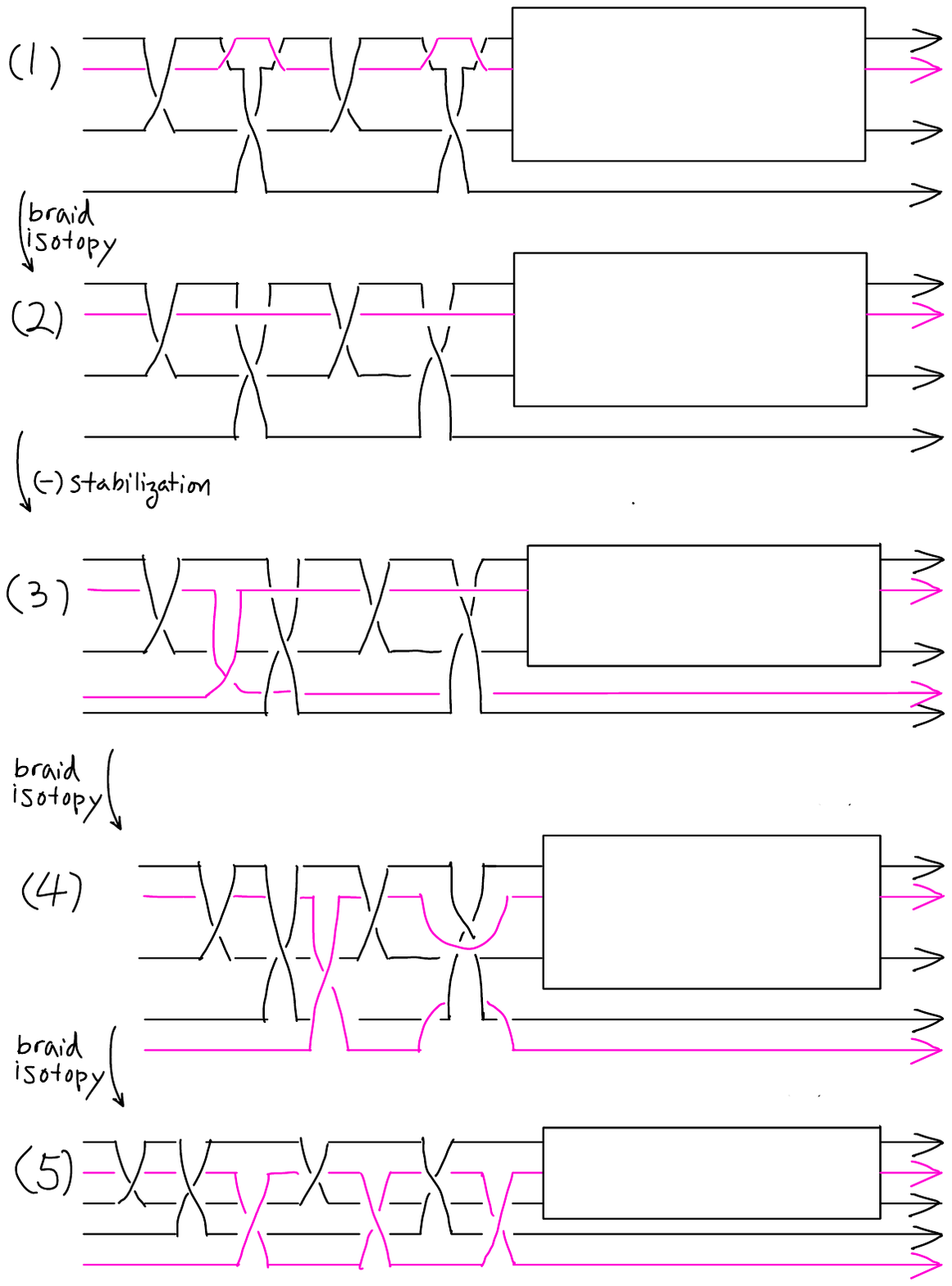}
 \caption{The effect of a negative tunnel-stabilization.}
\label{fig:stabilization-sequence}
\end{figure}
Consider a braid in Figure~\ref{fig:stabilization-sequence} (1), where the box contains some braiding. 
We may assume the braid bounds a Seifert surface $\Sigma$ containing a tunnel (Figure~\ref{fig:tunnel3} (1)) through witch other part of the Seifert surface (depicted by the green arrow) goes.  The idea of such a tunnel surface came from Ko and Lee's work in \cite{KoLee}. 
As in the passage $(1) \rightarrow \cdots \rightarrow (6)$ of Figure~\ref{fig:tunnel3} we can locally isotope the surface. 
\begin{figure}[h]
 \centering
\includegraphics[width=10cm]{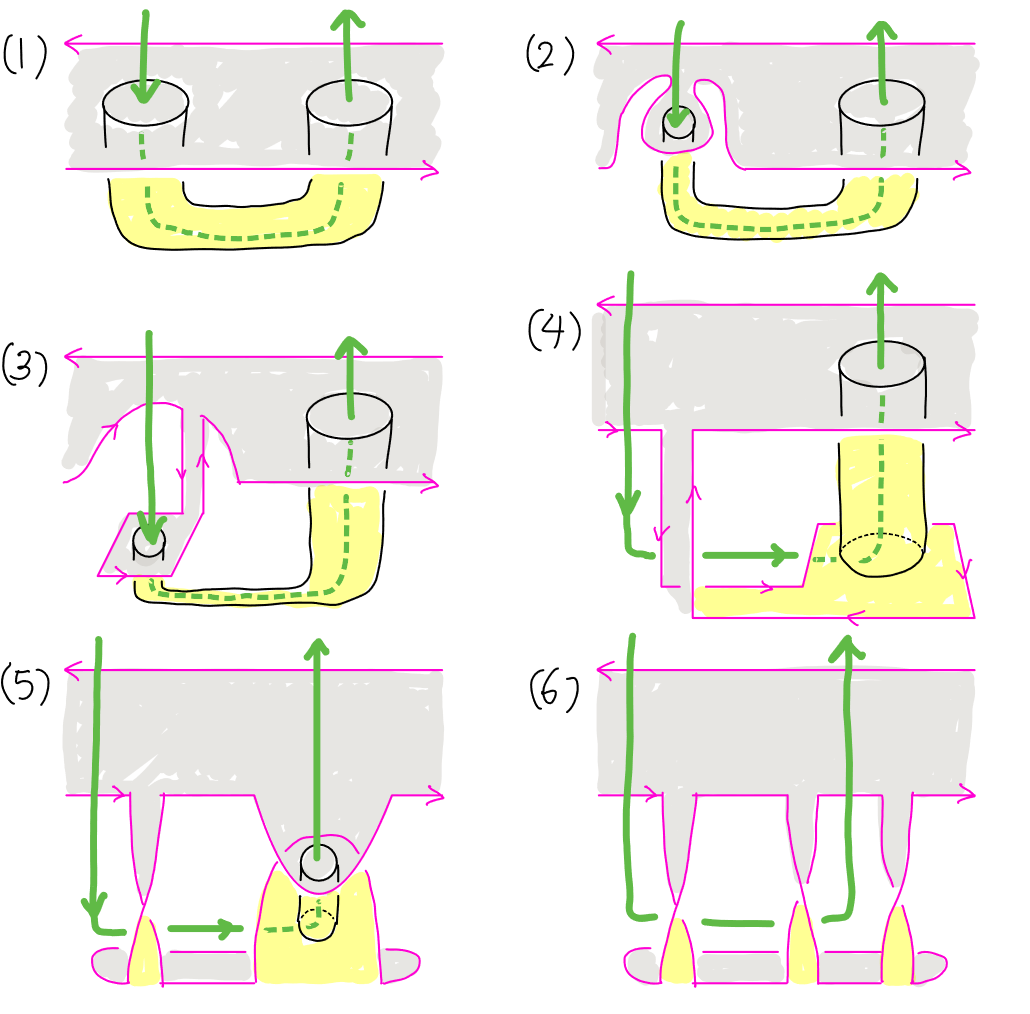}
 \caption{From (1) to (3) the pink strand is pinched and dragged through the tunnel. From (3) to (4) the strand is flipped performing a Reidemeister I move introducing the stabilization. From (4) to (5) the tunnel is being shortened performing a Reidemeister move II between the top and bottom strand. From (5) to (6) the tunnel is represented by two bands positive and negative bands.}
\label{fig:tunnel3}
\end{figure}

The effect of the isotopy on the boundary $\partial \Sigma$ can be understood as a negative stabilization and braid isotopy as shown in Figure~\ref{fig:stabilization-sequence}. 
We call it a {\em negative tunnel-stabilizaion}. 
In stead of using the {\em left} hole of the tunnel (Figure~\ref{fig:tunnel3} (2)), if we use the {\em right} hole we have a positive stabilization, which we call a {\em positive tunnel-stabilzation}. 

Note that a negative tunnel-stabilization preserves the number of negative bands however it decreases the number of positive bands by $1$. 
(Compare  Figure~\ref{fig:stabilization-sequence} (1) and (5).)
Thus the total number of bands decreases by one. 
This observation gives us the following: 

\begin{proposition}\label{prop:key-obs}
Let $W_\pm$ be the braid word obtained by a positive/negative tunnel-stabilization of $W$. 
\begin{itemize} 
\item 
A positive tunnel-stabilization decreases the negative band number  by $1$, and a negative stabilization preserves it; 
$$\nb(W) = \nb(W_+) +1 = \nb(W_-).$$ 
\item 
A tunnel-stabilization increases the Euler characteristic of the Bennequin surface by $2$. Namely, 
$$\chi(F_{W_-})= \chi(F_{W_+})=\chi(F_W)+2.$$
\end{itemize}
\end{proposition}

\begin{theorem}\label{prop:stabilization example}
Let $n\geq 1$. There exists a sequence of links $\{ L_n \}_{n\geq 1}$ such that it requires $n$ stabilizations for a minimal braid representative $\beta_n$ to achieve the negative band number $\nb(L_n)$ and the defect $\D(L_n)$. 
\end{theorem}

\begin{proof}
Let 
$\gamma_j:=  (a_{1,2}\ a_{2, 3+j}\ a_{1,2}^{-1}\ a_{1,3}^{-1})^2,$
see Figure~\ref{fig:sequence}. 
For $n=1, 2, \dots,$ consider the braid $\beta_n \in B_{n+3}$ represented by the word 
$$W_n := a_{1,3}^{-1}\ \gamma_1\ \gamma_2 \dots \gamma_n.$$
Let $L_n$ denote the link type of the braid closure of $\beta_n$. 
\begin{figure}[h]
 \centering
\includegraphics*[width=10cm, bb=15 200 440 780]{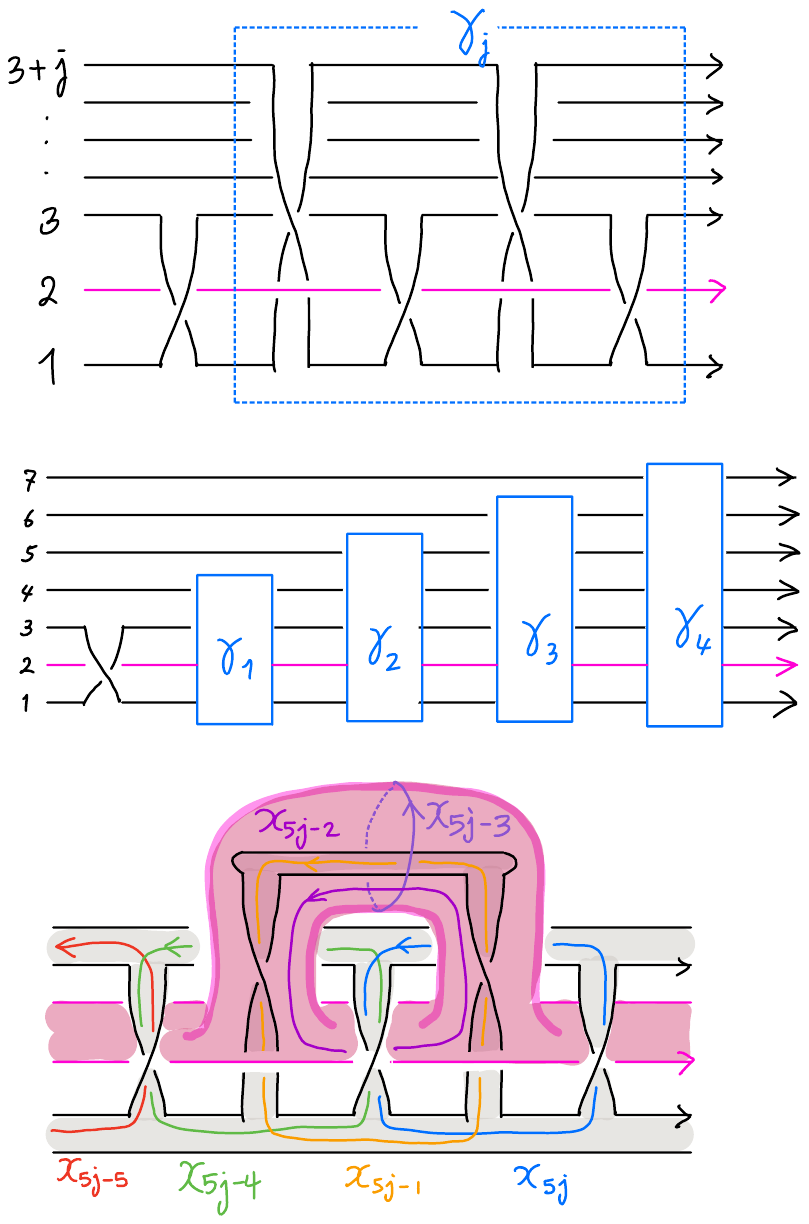}
 \caption{(Top) The word $\gamma_j$. 
 (Middle) The braid $\beta_4$.
 (Bottom) Part of the Seifert surface $\Sigma_n$ of $L_n$ involving $\gamma_j$.}
\label{fig:sequence}
\end{figure}

\begin{claim}\label{claim:braid index}
The  braid index of $L_n$ is $n+3$. 
Thus, the braid representative $\beta_n$ realizes the braid index.
\end{claim}

Note that $L_n$ contains an unknot component $U$ (colored pink in Figure~\ref{fig:sequence}) and the  component denoted $K_n = L_m \setminus U$.
It is enough to show that the braid index, $i(K_n)$, of $K_n$ is $n+2$. 
To this end let
$$ \gamma'_j := (\ a_{1,2+j}\ a_{1,2}^{-1}\ )^2$$
and consider the braid $\beta'_n\in B_{n+2}$ represented by the word
$$ W_n':= a_{1,2}^{-1}\ \gamma'_1\ \gamma'_2 \dots \gamma'_n.$$
Then the braid closure $\widehat{\beta_n'}$ yields $K_n$. 
Clearly $i(K_n) \leq n+2$. 

To study the lower bound of $i(K_n)$, we use the Morton-Franks-Williams inequality \cite{FW, Morton-seifert}. 
Consider the HOMFLY-PT polynomial $P(K)=P_K(v,z)$ of a link $K$ defined by the normalization $P({\rm unknot})=1$ and the skein relation 
$$ v^{-1}P(K_+) + vP(K_-) = zP(K_0).$$ 
Here $K_\epsilon$ is the link type of a diagram of $K$ whose particular crossing is locally changed according to the pictures in Figure~\ref{fig:HOMFLY}. 
Let $d^+(K)$ (resp. $d^-(K)$) be the largest (resp. smallest) power of $v$ in $P(K)$. 
The Morton-Franks-Williams inequality \cite{FW, Morton-seifert} states that
$$
\tfrac{1}{2}(d^+(K)- d^-(K))+1 \leq i(K).
$$

\begin{figure}[h]
 \centering
\includegraphics[width=10cm]{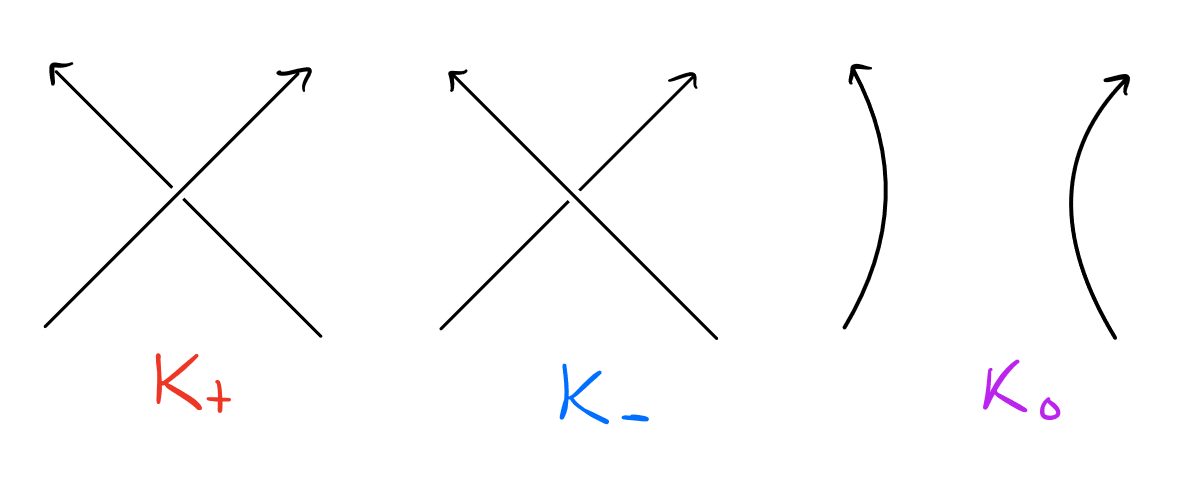}
 \caption{Local pictures of diagrams $K_+$, $K_-$ and $K_0$.}
\label{fig:HOMFLY}
\end{figure}
\begin{figure}[h]
 \centering
\includegraphics[width=10cm]{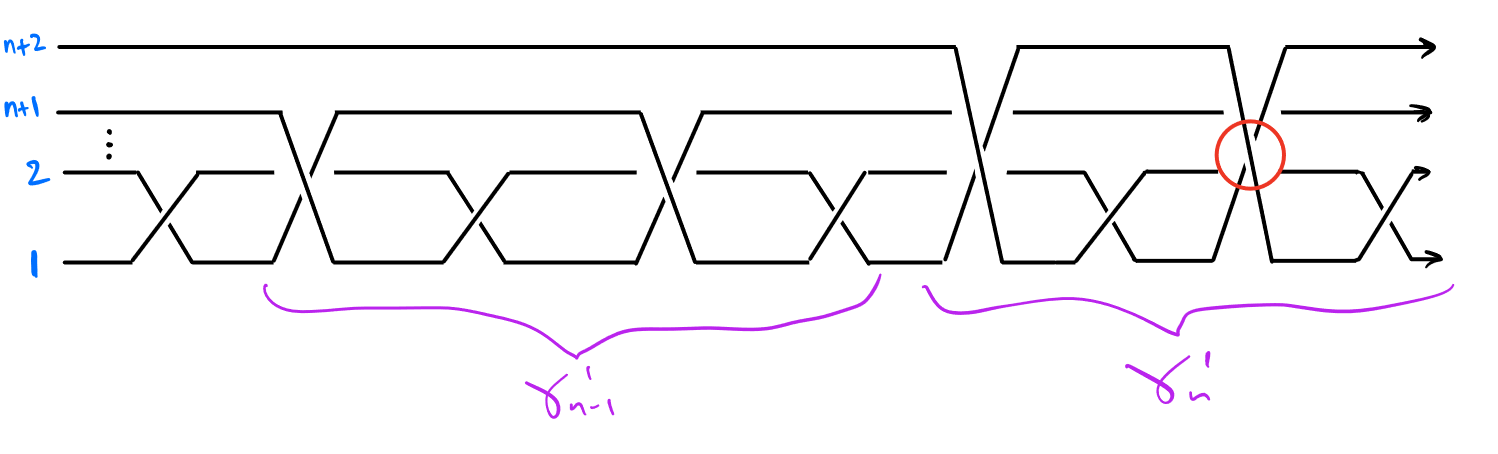}
 \caption{The braid $\beta_n'$. We set $K_+=K_n = \widehat{\beta_n'}$.}
\label{fig:HOMFLY-crossing}
\end{figure}

In order to compute the HOMFLY-PT polynomial of $K_n$, we focus on the positive crossing circled in Figure~\ref{fig:HOMFLY-crossing} and view $K_n$ as $K_+$. 
The resulting $K_-$ and $K_0$ are depicted in Figures~\ref{fig:HOMFLY-negative} and \ref{fig:HOMFLY-zero}, respectively.

\begin{figure}[h]
 \centering
\includegraphics[width=10cm]{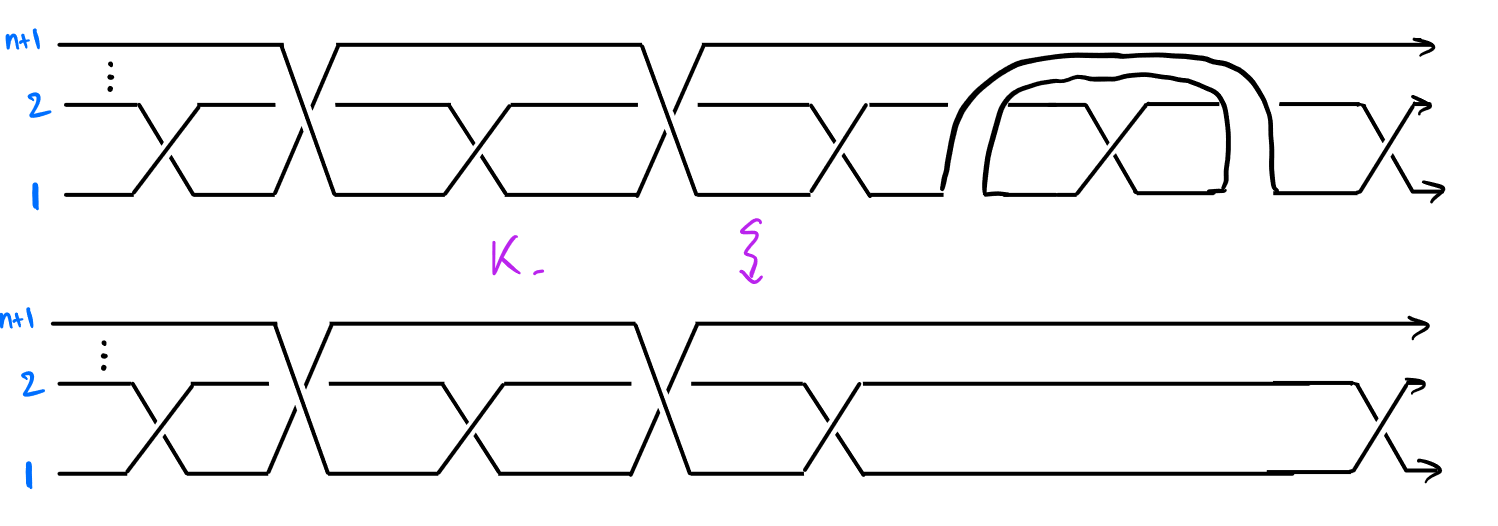}
\caption{The negative-resolution $K_-= \widehat{\beta_{n-1}' \sigma^{-1}}$.}
\label{fig:HOMFLY-negative}
\end{figure}

\begin{figure}[h]
\centering
\includegraphics[width=10cm]{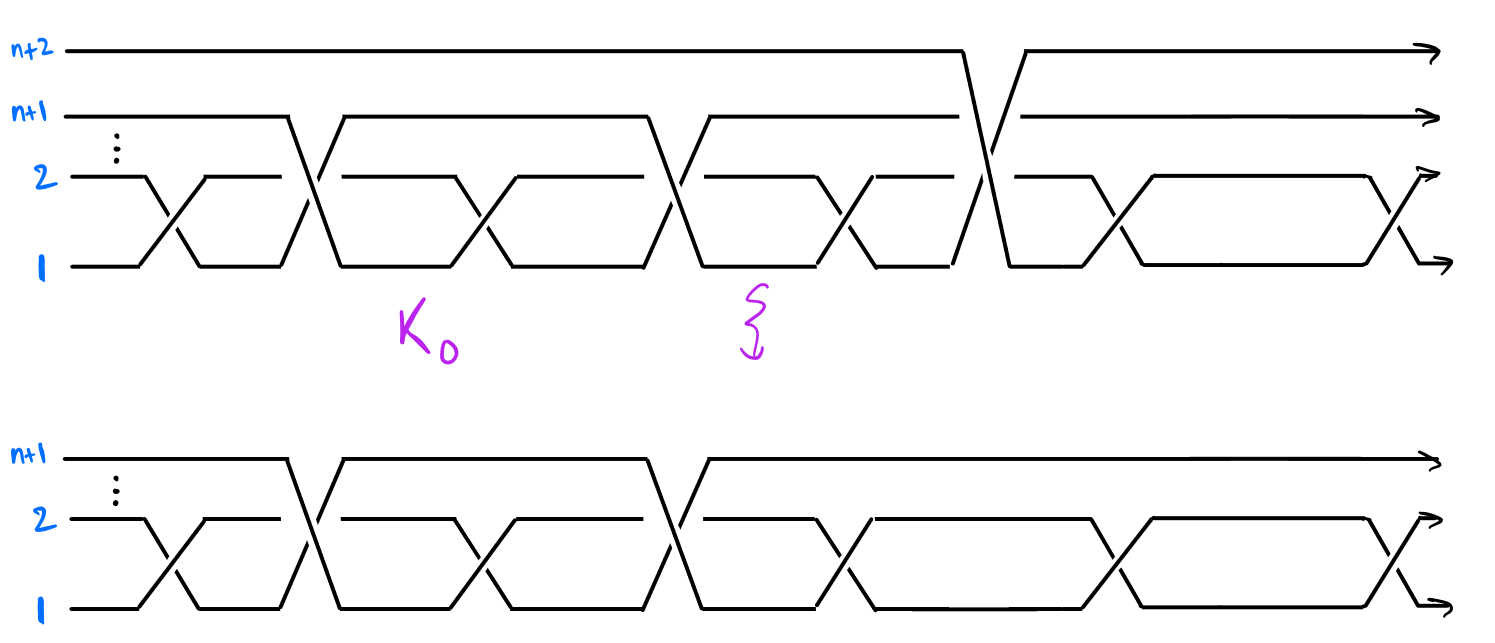}
\caption{The zero-resolution $K_0= \widehat{\beta_{n-1}' \sigma^{-2}}$.}
\label{fig:HOMFLY-zero}
\end{figure}

Let $\sigma= a_{1,2}$. Applying the skein relation we obtain the following:
$$ P(K_n) = v^2 P(\widehat{\beta_{n-1}' \sigma^{-1}}) + vz P(\widehat{\beta_{n-1}' \sigma^{-2}})$$
Recursively applying the skein relation we obtain a general formula:
$$ P(K_n)= v^n \sum_{i=1}^n {n\choose i} v^{n-i}z^i P(\widehat{\sigma^{-n-1-i}})$$
It is a straight-forward calculation to verify that $d^+(\widehat{\sigma^{-k}}) = -k+1$ and $d^-(\widehat{\sigma^{-k}}) =-k-1$. Plugging this in we get
\begin{eqnarray*}
d^+(K_n) &=& n + (n-0) + d^+(\widehat{\sigma^{-n-1}}) = 2n - n = n,\\
d^-(K_n) &=& n + (n-n) + d^-(\widehat{\sigma^{-2n-1}}) = n + (-2n-2)= - n-2.
\end{eqnarray*}
Thus we obtain the following lower bound of the braid index for $K_n$:
$$ \tfrac{1}{2}(d^+(K_n) - d^-(K_n)) + 1 = \tfrac{1}{2}(2n +2)+ 1 = n+2 \leq i(K_n)$$
This concludes Claim~\ref{claim:braid index}.

Consider a Seifert surface, $\Sigma_n$, of $L_n$ assembled by the pieces depicted in the bottom sketch in Figure \ref{fig:sequence}. It has $\chi(\Sigma_n)= -5n +2$. 

\begin{claim}\label{second claim}
$\chi(L_n)= -5n+2$.
\end{claim}

Pick an appropriate basis $\{x_1, \dots, x_n\}$ of $H_1(\Sigma_{5n})$ according to the bottom sketch in Figure~\ref{fig:sequence}. We then compute the associated Seifert matrix 
$$
A_n := \begin{bmatrix}
B & C & 0 &  0 & \dots & 0\\
0 & B & C &  0       &  \dots & 0\\
0 & 0 & B & C &\ddots & \vdots \\
0 & \vdots & 0 & \ddots & \ddots & 0 \\
\vdots & \vdots & \vdots & \ddots & B & C\\
0 & 0 & 0 & \dots & 0 & B
\end{bmatrix} $$ of size $5n \times 5n$, where
$$ B:=\begin{bmatrix}
1 & 0 & -1 & -1 & -1\\
0 & 0 & 0 & 1 & 0\\
-1 & 1 & 0 & 0 & 1 \\
0& 1& 0 & -1& 0\\
0& 0& 1& 1& 1
\end{bmatrix}
\quad \mbox{ and } \quad 
C:= 
\begin{bmatrix}
0&0&0&0&0\\
0&0&0&0&0\\
0&0&0&0&0\\
0&0&0&0&0\\
-1&0&0&0&0\\
\end{bmatrix}. 
$$

Recall the Alexander polynomial of $L_n$ is $\Delta_{L_n}(t)= \det(A_n - tA_n^T)$, where $A_n^T$ is the transpose of $A_n$. 
We claim that the degree-span of $\Delta_{L_n}(t)$ is $ 5n$, which implies that 
any Seifert surface $S$ of $L_n$ has 
$5n \leq \rk (H_1(S))$. 
Since our surface $\Sigma_n$ has $\rk(H_1(\Sigma_n))=5n$ it means that $\Sigma_n$ realizes the Euler characteristic $\chi(L_n)$. 
That is, 
$$ \chi(L_n) = \chi(\Sigma_n)= -5n + 2.$$

The lowest exponent occurring in $\Delta_{L_n}(t)$ is zero since 
$A_n$ is a block upper triangular matrix and
$
\Delta_{L_n}(0) = \det(A_n)= (\det(B))^n= 1^n=1.
$
On the other hand, the highest degree term of $\Delta_{L_n}(t)$ is exactly $
\det(-t A_n^T)=(-t)^{5n}\det(A_n)=(-t)^{5n}.
$
Thus the highest degree of $\Delta_{L_n}(t)$ is $5n$. 
This concludes Claim~\ref{second claim}.

Since the surface $\Sigma_n$ has $n$ tunnels we can apply positive tunnel-stabilization $n$ times. 

\begin{claim}\label{third claim}
After the $n$ positive tunnel-stabilizations of $W_n\in \beta_n\in B_{n+3}$ the resulting braid word, $\widetilde{W_n} \in \widetilde{\beta_n} \in B_{2n+3}$,
achieves the defect $\D(L_n)$ and the negative band number $\nb(L_n)$, 
which implies $\widetilde{W_n}$ also achieves $\chi(L_n)$ and $\SL(L_n)$. 
\end{claim}

The initial braid word $W_n = a_{1,3}^{-1}\ \gamma_1\ \gamma_2 \dots \gamma_n \in \beta_n$ contains $4$ negative bands in each $\gamma_j$ therefore $\nb(L_n)\leq \nb(\beta_n)\leq \nb(W_n)=4n+1$. 
According to Proposition~\ref{prop:key-obs} applying $n$ positive tunnel-stabilizations deletes $n$ negative bands which yields the following upper bound
$$\nb(L_n) \leq \nb(\widetilde{\beta_n}) \leq \nb(\widetilde{W_n})=4n+1- n = 3n+1.$$

Since $\beta_n$ achieves the braid index of $L_n$, the truth of the generalized Jones conjecture \cite{DynnikovPrasolov, Menasco-LaFountain, Keiko06} implies that $\beta_n$ realizes the maximal self-linking number; that is 
$$\SL (L_n) = \sel (\beta_n) = a(\beta_n) - n(\beta_n) = -(n+4),$$ 
where $a(\beta_n)$ denotes the algebraic lenth and $n(\beta_n)$ denotes the number of strands.  
Therefore
$$\D(L_n)= \tfrac{1}{2}(-\chi(L_n) -\SL(L_n)) = \tfrac{1}{2}(-(-5n+2)+(n+4))= 3n +1.$$
Since $\D(L_n) \leq \nb(L_n)$ by (\ref{eq:inequality}) we obtain 
$$\D(L_n) =\nb(L_n)=\nb(\widetilde{W_n})=3n+1.$$
This concludes Claim~\ref{third claim}, thus Theorem~\ref{prop:stabilization example}.
\end{proof}

With these examples in mind here is a refinement of Ito and Kawamuro's conjecture:

\begin{conjecture}
Suppose that a word $W$ realizes $\nb(K)$. Then $W$ also realizes $\chi(K)$ and $\SL(K)$. 
\end{conjecture}
\begin{remark}\label{final-remark}
Recall that each tunnel admits both positive and negative tunnel-stabilization. 
In Claim~\ref{third claim} 
we applied positive tunnel-stabilizations $n$ times to the Seifert surface $\Sigma_n$ which has $n$ tunnels. 
Instead, we can also apply negative tunnel-stabilizations $n$ times to $\Sigma_n$ and call the resulting braid $\widetilde{W_n}'$. 
With some computation, we see $\widetilde{W_n}'$ realizes $\chi(L_n)$, but 
it does not realize either $\nb(L_n)$ or $\SL(L_n)$. 

Therefore, for a link $L$ and a braid word $W \in L$, the condition $\chi(F_W)=\chi(L)$ does not imply $\nb(W)=\nb(L)$ in general.  
However, in case of SQP links (i.e., $\nb(L)=0$), it is still open whether $\chi(F_W)=\chi(L)$ implies $\nb(W)=0$. 
\end{remark}

\nocite{*}
\bibliographystyle{amsplain}
\bibliography{bib.bib}
\end{document}